\documentclass[12pt]{amsart}

\usepackage{amsmath}
\usepackage{amsfonts}
\usepackage{amsthm}

\usepackage[cmtip,arrow]{xy}
\usepackage{pb-diagram, pb-xy}
\usepackage{pb-diagram, pb-xy}

\newtheorem{theorem}{Theorem}
\newtheorem{corollary}{Corollary}
\newtheorem{lemma}{Lemma}
\newtheorem{proposition}{Proposition}
\newtheorem{definition}{Definition}

\newtheorem{remark}{Remark}

\begin{document}
\title{Gorenstein projective precovers}
\author{Sergio Estrada, Alina Iacob, Katelyn Yeomans}

%

\maketitle %

\begin{abstract}
We prove that the class of Gorenstein projective modules is special precovering over any left GF-closed ring such that every Gorenstein projective module is Gorenstein flat and every Gorenstein flat module has finite Gorenstein projective dimension. This class of rings includes (strictly) Gorenstein rings, commutative noetherian rings of finite Krull dimension, as well as right coherent and left n-perfect rings. In section 4 we give examples of left GF-closed rings that have the desired properties (every Gorenstein projective module is Gorenstein flat and every Gorenstein flat has finite Gorenstein projective dimension) and that are not right coherent.
\end{abstract}


\section{introduction}

The class of Gorenstein projective modules is one of the key elements in Gorenstein homological algebra. So it is natural to consider the question of the existence of the Gorenstein projective precovers.\\
The existence of the Gorenstein projective prevovers over Gorenstein rings is known (Enochs-Jenda, 2000). Then J{\o}rgensen proved their existence over commutative noetherian rings with dualizing complexes (2007). More recently (2011), Murfet and Salarian proved the existence of the Gorenstein projective precovers over commutative noetherian rings of finite Krull dimension. In \cite{iacob:14:gor.proj.precovers} we extended their result: we proved that if $R$ is a right coherent and left n-perfect ring, then the class of Gorenstein projective complexes is special precovering in the category of unbounded complexes, $Ch(R)$. As a corollary we obtained the existence of the special Gorenstein projective precovers in $R-Mod$ over the same type of rings. 

We prove that the class of Gorenstein projective modules is special precovering over any left GF-closed ring $R$ such that every Gorenstein projective module is Gorenstein flat and every Gorenstein flat module has finite Gorenstein projective dimension. This class of rings includes that of right coherent and left n-perfect rings. But the inclusion is a strict one: in section 4 we give examples of left GF-closed rings that have the desired properties (every Gorenstein projective is Gorenstein flat and every Gorenstein flat has finite Gorenstein projective dimension), and that are \textbf{not} right coherent.

\section{preliminaries}
Throughout the paper $R$ will denote an associative ring with identity. Unless otherwise stated, by \emph{module} we mean \emph{left} $R$-module.

We will denote by $Proj$ the class of all projective modules.
We recall that an $R$-module $M$ is Gorenstein projective if there exists an exact and $Hom(-, Proj)$ exact complex of projective modules\\ $\textbf{P}= \ldots \rightarrow P_1 \rightarrow P_0 \rightarrow P_{-1} \rightarrow \ldots $ such that $M = Ker (P_0 \rightarrow P_{-1})$.\\
We will use the notation $\mathcal{GP}$ for the class of Gorenstein projective modules.\\

We also recall the definitions for Gorenstein projective precovers, covers, and special precovers. \\
\begin{definition} A homomorphism $\phi: G \rightarrow M$ is a Gorenstein projective precover of $M$ if $G$ is Gorenstein projective and if for any Gorenstein projective module $G'$ and any $\phi' \in Hom(G', M)$ there exists $u \in Hom(G', G)$ such that $\phi' = \phi u$. \\
A Gorenstein projective precover $\phi$ is said  to be a cover if any $v \in End_R(G)$ such that $\phi v = \phi$ is an automorphism of $G$.\\
A Gorenstein projective precover $\phi$ is said to be \textbf{special} if $Ker (\phi) $ is in the right orthogonal class of that of Gorenstein projective modules,\\ $\mathcal{GP} ^\bot = \{L, Ext^1(G',L)=0 $ for all Gorenstein projective modules $G' \}$.
\end{definition}

The importance of the Gorenstein projective (pre)covers comes from the fact that they allow defining the Gorenstein projective resolutions: if the ring $R$ is such that every $R$-module $M$ has a Gorenstein projective precover then for every $M$ there exists a $Hom(\mathcal{GP}, -)$ exact complex $\ldots \rightarrow G_1 \rightarrow G_0 \rightarrow M \rightarrow 0$ with $G_0 \rightarrow M$ and $G_i \rightarrow Ker( G_{i-1} \rightarrow G_{i-2})$ Gorenstein projective precovers. Such a complex is called a Gorenstein projective resolution of $M$; it is unique up to homotopy so it can be used to compute right derived functors of $Hom$.\\


We also use Gorenstein flat modules. They are defined in terms of the tensor product:\\
\begin{definition}
A module $G$ is Gorenstein flat if there exists an exact complex of flat modules  ${\rm {\bf F }}= \ldots \rightarrow F_1 \rightarrow F_0 \rightarrow F_{-1} \rightarrow \ldots $ such that $I \otimes {\rm {\bf F }}$ is still exact for any injective right $R$-module $I$, and such that $G = Ker (F_0 \rightarrow F_{-1})$.\\
\end{definition}

We will use $\mathcal{GF}$ to denote the class of Gorenstein flat modules.\\
The Gorenstein flat precovers, covers and resolutions are defined in a similar manner to that of the Gorenstein projective ones (simply replace $\mathcal{GP}$ with $\mathcal{GF}$ in the definition).\\

\section{main result}

We recall that a ring $R$ is left GF-closed if the class of Gorenstein flat left $R$-modules is closed under extensions. In this case, $\mathcal{GF}$ is a projectively resolving class (by \cite{bennis:07:stronglygorenstein}).

Our main result is the existence of special Gorenstein projective precovers over any left GF-closed ring such that every Gorenstein projective is Gorenstein flat and every Gorenstein flat has finite Gorenstein projective dimension. \\

We will use the following:\\

\begin{proposition}
Every module of finite Gorenstin projective dimension has a special Gorenstein projective precover.
\end{proposition}

\begin{proof}
Let $G$ be a module of finite Gorenstein projective dimension, $G.p.d. G = d < \infty$. Then there exists an exact sequence $0 \rightarrow C \rightarrow P_{d-1} \rightarrow \ldots \rightarrow P_0 \rightarrow G \rightarrow 0$ with all $P_i$ projective modules and with $C$ Gorenstein projective.\\
Since $C \in \mathcal{GP}$ there is an exact and $Hom(-, Proj)$ exact sequence $0 \rightarrow C \rightarrow T_{d-1} \rightarrow \ldots \rightarrow T_0 \rightarrow D \rightarrow 0$ with each $T_j$ projective and with $D$ a Gorenstein projective module.\\
The fact that each $P_i$ is projective allows constructing a commutative diagram:\\

\[
\begin{diagram}
\node{0}\arrow{e}\node{C}\arrow{s,=}\arrow{e}\node{T_{d-1}}\arrow{s}\arrow{e}\node{\cdots}\arrow{e}\node{T_0}\arrow{s}\arrow{e}\node{D}\arrow{s}\arrow{e}\node{0}\\
\node{0}\arrow{e}\node{C}\arrow{e}\node{P_{d-1}}\arrow{e}\node{\cdots}\arrow{e}\node{P_0}\arrow{e}\node{G}\arrow{e}\node{0}
\end{diagram}
\]

This gives an exact sequence $ 0 \rightarrow T_{d-1} \rightarrow P_{d-1} \oplus T_{d-2} \rightarrow \ldots \rightarrow P_1 \oplus T_0 \rightarrow P_0 \oplus D \xrightarrow{\delta} G \rightarrow 0$. \\
Let $W=Ker(\delta)$. The exact sequence $ 0 \rightarrow T_{d-1} \rightarrow P_{d-1} \oplus T_{d-2} \rightarrow \ldots \rightarrow P_1 \oplus T_0 \rightarrow W \rightarrow 0$ with all $T_i$ and all $P_j$ projective modules gives that $W$ has finite projective dimension, so $W \in \mathcal{GP}^\bot$. The exact sequence $0 \rightarrow W \rightarrow P_0 \oplus D \xrightarrow{\delta} N \rightarrow 0$ with $P_0 \oplus D$ Gorenstein projective and with $W \in \mathcal{GP}^\bot$ shows that $\delta$ is a special Gorenstein projective precover.
\end{proof}

Our main result is the following:\\

\begin{theorem}
Let $R$ be a left GF-closed ring. If every Gorenstein projective module is Gorenstein flat and every Gorenstein flat $R$-module has finite Gorenstein projective dimension then the class of Gorenstein projective modules is special precovering in $R-Mod$.
\end{theorem}

\begin{proof}
Let $_RX$ be any left $R$-module. Since $R$ is left GF-closed, the class of Gorenstein flat modules is covering in $R-Mod$ (\cite[Corollary 3.5]{yangliu}). So there exists an exact sequence $0 \rightarrow Y \rightarrow N \rightarrow X \rightarrow 0$ with $N$ Gorenstein flat and with $Y \in \mathcal{GF}^\bot \subset \mathcal{GP}^\bot$ (because we have that $\mathcal{GP} \subset \mathcal{GF}$). Since $N$ has finite Gorenstein projective dimension, by Proposition 1, there is an exact sequence $0 \rightarrow W \rightarrow T \rightarrow N \rightarrow 0$ with $T$ Gorenstein projective and $W \in \mathcal{GP}^\bot$.

Form the pull back diagram:\\

\[
\begin{diagram}
\node{}\node{0}\arrow{s}\node{0}\arrow{s}\\
\node{}\node{W}\arrow{s}\arrow{e,=}\node{W}\arrow{s}\\
\node{0}\arrow{e}\node{A}\arrow{s}\arrow{e}\node{T}\arrow{s}\arrow{e}\node{X}\arrow{s,=}\arrow{e}\node{0}\\
\node{0}\arrow{e}\node{Y}\arrow{e}\node{N}\arrow{e}\node{X}\arrow{e}\node{0}
\end{diagram}
\]

The exact sequence $0 \rightarrow W \rightarrow A \rightarrow Y \rightarrow 0$ with $W$, $Y \in \mathcal{GP}^\bot$ gives $A \in \mathcal{GP}^\bot$. So we have an exact sequence $0 \rightarrow A \rightarrow T \rightarrow X \rightarrow 0$ with $T \in \mathcal{GP}$ and $A \in \mathcal{GP}^\bot$. It follows that $T \rightarrow X$ is a special Gorenstein projective precover of $X$.

\end{proof}

\begin{corollary}
Let $R$ be a left GF-closed ring such that $\mathcal{GP} \subseteq \mathcal{GF}$ and every Gorenstein flat module has finite Gorenstein projective dimension. Then $(\mathcal{GP}, \mathcal{GP}^\bot)$ is a complete hereditary cotorsion pair.
\end{corollary}

\begin{proof}
- We prove first that $(\mathcal{GP}, \mathcal{GP}^\bot)$ is a cotorsion pair.\\
Let $X \in ^\bot\!\! (\mathcal{GP}^\bot)$. By Theorem 1 there exists an exact sequence $0 \rightarrow A \rightarrow B \rightarrow X \rightarrow 0$ with $B$ Gorenstein projective and with $A \in \mathcal{GP}^\bot$. Then $Ext^1(X,A)=0$, so the sequence is split exact. Since $B \simeq A \oplus X$ it follows that $X$ is Gorenstein projective. Thus $^\bot (\mathcal{GP}^\bot) = \mathcal{GP}$.\\
- The pair $(\mathcal{GP}, \mathcal{GP}^\bot)$ is complete by Theorem 1.\\
- Since the class of Gorenstein projective modules is projectively resolving (\cite[Theorem 2.5]{holm04}) the pair $(\mathcal{GP}, \mathcal{GP}^\bot)$ is hereditary.

\end{proof}

\textbf{Examples} of left GF-closed rings such that every Gorenstein projective is Gorenstein flat and every Gorenstein flat has finite Gorenstein projective dimension:\\
1. Gorenstein rings\\
2. commutative noetherian rings of finite Krull dimension: by \cite{christensen:06:gorenstein} over such a ring every Gorenstein projective is Gorenstein flat and every Gorenstein flat has finite Gorenstein projective dimension\\

We prove next that every right coherent and left n-perfect ring also satisfies Theorem 1.\\

We recall that a ring $R$ is right coherent if every direct product of flat left $R$-modules is a flat module. We consider right coherent rings such that every flat left $R$-module has finite projective dimension. In this case there exists an integer $n \ge 0$ such that $pd_R F \le n$ for any flat $R$-module $F$. Such a ring $R$ is called a left n-perfect ring.\\

In order to prove that every such ring satisfies the hypotheses of Theorem 1, we will need to argue that over these rings every Gorenstein flat module has finite Gorenstein projective dimension.\\

We give an equivalent characterization below (Proposition 2) of the condition that $G.p.d. G < \infty$ for any Gorenstein flat module $G$. It uses the following Lemma:\\

\begin{lemma}
Let $R$ be a left n-perfect ring. If $F$ is a flat $R$-module then there exists an exact sequence $0 \rightarrow F \rightarrow S^0 \rightarrow S^1 \rightarrow \ldots \rightarrow S^n \rightarrow 0$ with all $S^j$ flat and cotorsion modules.
\end{lemma}

\begin{proof}
Since $(Flat, Cotorsion)$ is a complete cotorsion pair, there exists a short exact sequence $0 \rightarrow F \rightarrow S^0 \rightarrow F^0 \rightarrow 0$ with $S^0$ cotorsion and with $F^0$ a flat module. And since $F$ is flat, it follows that $S^0 $ is both flat and cotorsion. Similarly, there exists an exact sequence $0 \rightarrow F^0 \rightarrow S^1 \rightarrow F^1 \rightarrow 0$ with $S^1$ flat and cotorsion and $F^1$ flat. Continuing, we obtain an exact sequence $0 \rightarrow F \rightarrow S^0 \rightarrow \ldots \rightarrow S^{n-1} \rightarrow F^n \rightarrow 0$ with all $S^i$ flat and cotorsion and with $F^n$ flat. \\
We show that $F^n$ is also cotorsion. Let $K$ be a flat module. Since $pd_R K \le n$, we have that $Ext^{n+1}(K,F)=0$. And since all $S^i$ are flat and cotorsion, we have that $Ext^{n+1}(K,F) \simeq Ext^1(K,F^n)$. So $Ext^1(K,F^n) = 0$ for all flat $R$-modules $K$, therefore $F^n$ is cotorsion.
\end{proof}

We can prove now:\\
\begin{proposition}
Let $R$ be a left GF-closed and left n-perfect ring. The following are equivalent:\\
1. $Gpd_R G \le n$ for any Gorenstein flat module $G$.\\
2. $Gpd_R G <\infty$ for any Gorenstein flat module $G$.\\
3. $Ext^i(G,F) = 0$ for any Gorenstein flat module $G$, any flat and cotorsion module $F$ and all $i \ge 1$

\end{proposition}

\begin{proof}
1. $\Rightarrow$ 2. is immediate.\\
 2. $\Rightarrow$ 3. Let $F$ be flat and cotorsion and let $G'$ be a Gorenstein flat $R$-module. Then there exists a strongly Gorenstein flat module $G$ such that $G'$ is a direct summand of $G$ (by \cite{bennis:07:stronglygorenstein}). Since there exists an exact sequence $0 \rightarrow G \rightarrow K \rightarrow G \rightarrow 0$ with $K$ flat it follows that $Ext^i(G,F) \simeq Ext^1(G,F) $ for all $i \ge 1$. And since $Gpd_R G < \infty$ and $Flat \subset \mathcal{GP}^\bot$, there exists $l$ such that $Ext^j(G,F) = 0$ for any $j \ge l+1$. By the above, $Ext^i(G,F) = 0$ for all $i \ge 1$. Since $Ext^i(G', F)$ is a direct summand of $Ext^i(G,F)=0$ it follows that $Ext^i(G',F)=0$ for all $i \ge 1$.\\

3. $\Rightarrow$ 1. Let $G \in \mathcal{GF}$. Then there exists an N-totally acyclic complex $N$ such that $G = Z_0(N)$. Consider a partial projective resolution of $N$: $0 \rightarrow C \rightarrow P_{n-1} \rightarrow \ldots \rightarrow P_0 \rightarrow N \rightarrow 0$. Then $C$ is an exact complex. Since for each j we have an exact sequence $0 \rightarrow C_j \rightarrow P_{{n-1},j} \rightarrow \ldots \rightarrow P_{0,j} \rightarrow N_j \rightarrow 0$ with all $P_{i,j}$ projective and since $pd_RN_j \le n$ it follows that $C_j$ is projective for all j.\\
Also for each j there is an exact sequence $0 \rightarrow Z_j(C) \rightarrow Z_j(P_{n-1}) \rightarrow \ldots \rightarrow Z_j(P_0) \rightarrow Z_j(N) \rightarrow 0$. Since $Z_j(N)$ is Gorenstein flat and $Z_j(P_i)$ is projective for all i and since the ring $R$ is left GF-closed, it follows that $Z_j(C)$ is Gorenstein flat for all i.\\
We show that $C$ is $Hom(-, Flat)$ exact, and so all $Z_j(C)$ are Gorenstein projective modules.\\
Let $F$ be a flat module. Since $R$ is left n-perfect there exists an exact sequence $0 \rightarrow F \rightarrow S^0 \rightarrow \ldots \rightarrow S^n \rightarrow 0$ with all $S^i$ flat and cotorsion modules. By the hypothesis, we have $Ext^i(Z_j(C), S^t)=0$ for all i, all j, and all $0 \le t \le n$. Then $Ext^{i+n}(Z_j(C),F)=0$ for all $i \ge 1$. Since $C$ is a complex of projective modules there is also an exact sequence $0 \rightarrow Z_{j+n}(C) \rightarrow C_{j+n} \rightarrow \ldots \rightarrow C_{j+1} \rightarrow Z_j(C) \rightarrow 0$ with all $C_i$ projective. This gives that $Ext^{i+n}(Z_j(C), F) \simeq Ext^i(Z_{j+n}(C), F)$, so we obtain that $Ext^i(Z_{j+n}(C), F)=0$ for all j and for all $i \ge 1$. Then for $j \rightarrow j-n$ we obtain that $Ext^i(Z_j(C), F)=0$ for all $i \ge 1$. So $Hom(C,F)$ is exact for all flat modules $F$. In particular, $C$ is a totally acyclic complex of projective modules, so $Z_j(C)$ is Gorenstein projective for all j. The exact sequence $0 \rightarrow Z_j(C) \rightarrow Z_j(P_{n-1}) \rightarrow \ldots \rightarrow Z_j(P_0) \rightarrow Z_j(N) \rightarrow 0$ gives that $Gpd_R Z_j(N) \le n$ for all j.
\end{proof}

We can prove now:\\

\begin{theorem}
If $R$ is right coherent and left n-perfect then $(\mathcal{GP}, \mathcal{GP}^\bot)$ is a complete hereditary cotorsion pair.
\end{theorem}

\begin{proof}
It is known that every right coherent ring is left GF-closed (\cite{bennis:07:stronglygorenstein}). It is also known that if $R$ is right coherent and left n-perfect then every Gorenstein projective module is Gorenstein flat (\cite{christensen:06:gorenstein}). So in order to prove that a right coherent and left n-perfect ring satisfies Theorem 1 it suffices to check that every Gorenstein flat module has finite Gorenstein projective dimension. By Lemma 1 this is equivalent with showing that $Ext^i(G,F)=0$ for all $i \ge 1$, for any Gorenstein flat module $G$ and any flat and cotorsion module $F$.\\

Let $F$ be flat and cotorsion. Consider the pure exact sequence $$0 \rightarrow F \rightarrow F^{++} \rightarrow Y \rightarrow 0$$
Since $F$ is flat and $R$ is right coherent, the module $F^{++}$ is also flat. Since the sequence is pure exact it follows that $Y$ is also flat. Then since $F$ is cotorsion we have $Ext^1(Y,F)=0$, so $F^{++} \simeq F \oplus Y$, and therefore $Ext^i(G, F^{++}) \simeq Ext^i(G,F) \oplus Ext^i(G,Y)$ for all $i \ge 1$.\\
We have $Ext^i(G, F^{++}) \simeq Ext^i(F^+, G^+)$. For a Gorenstein flat module $G$ its character module $G^+$ is Gorenstein injective (by \cite{holm04}), so we have that $Ext^i(F^+, G^+)=0$ for all $i \ge 1$ (because $F^+$ is injective). Thus $Ext^i(G, F^{++})=0$ and therefore $Ext^i(G,F)=0$ for all $i \ge 1$.\\

So any right coherent and left n-perfect ring $R$ satisfies Theorem 1. Then by Corollary 1, $(\mathcal{GP}, \mathcal{GP}^\bot)$ is a complete hereditary cotorsion pair.
\end{proof}

\begin{remark}
We already proved in \cite{iacob:14:gor.proj.precovers} that the class $\mathcal{GP}$ is special precovering over any right coherent and left n-perfect ring $R$. For completeness, we included a different proof here.
\end{remark}

\section{examples of non coherent rings that are left GF-closed, such that $\mathcal{GP} \subseteq \mathcal{GF}$, and every Gorenstein flat has finite Gorenstein projective dimension}

We proved in the previous section that the class of right coherent and left n-perfect rings is contained in that of rings satisfying Theorem 1. We show that this inclusion is a strict one.\\

1) Consider the ring $R$ below. 

$\displaystyle R = \left[
                     \begin{array}{ccc}
                       \mathbb{Q} & \mathbb{Q} & \mathbb{R} \\
                       0 & \mathbb{Q} & \mathbb{R} \\
                       0 & 0 & \mathbb{Q} \\
                     \end{array}
                   \right] /
                   \left[
                     \begin{array}{ccc}
                       0 & 0 & \mathbb{R} \\
                       0 & 0 & 0 \\
                       0 & 0 & 0 \\
                     \end{array}
                   \right]$

                   Its Jacobson radical is \\

                   $\displaystyle J(R) = \left[
                     \begin{array}{ccc}
                       0 & \mathbb{Q} & \mathbb{R} \\
                       0 & 0 & \mathbb{R} \\
                       0 & 0 & 0\\
                     \end{array}
                   \right] /
                   \left[
                     \begin{array}{ccc}
                       0 & 0 & \mathbb{R} \\
                       0 & 0 & 0 \\
                       0 & 0 & 0 \\
                     \end{array}
                   \right]$

                   Then we have that\\

                   $\displaystyle \frac{R}{J(R)} \simeq \left[
                     \begin{array}{ccc}
                       \mathbb{Q} & \mathbb{Q} & \mathbb{R} \\
                       0 & \mathbb{Q} & \mathbb{R} \\
                       0 & 0 & \mathbb{Q} \\
                     \end{array}
                   \right] /
                   \left[
                     \begin{array}{ccc}
                       0 & \mathbb{Q} & \mathbb{R} \\
                       0 & 0 & \mathbb{R} \\
                       0 & 0 & 0 \\
                     \end{array}
                   \right]$

                   This is isomorphic to the ring of diagonal matrices with entries from $\mathbb{Q}$, so isomorphic to $\mathbb{Q} \times \mathbb{Q} \times \mathbb{Q}$ which is semisimple. Also $J(R)^2 = 0$. So $R$ is semiprimary.\\

                   Since $ K = \left[ \begin{array}{ccc}
                       0 & 0 & \mathbb{R} \\
                       0 & 0 & 0 \\
                       0 & 0 & 0 \\
                     \end{array}
                     \right]$

is a two sided ideal of the ring $A = \left[ \begin{array}{ccc}
                       \mathbb{Q} & \mathbb{Q} & \mathbb{R} \\
                       0 & \mathbb{Q} & \mathbb{R} \\
                       0 & 0 & \mathbb{Q} \\
                     \end{array}
                     \right]$

                     it follows that $gl.dim (R) = gl. dim. (A/K) \le 2$ (by \cite{harada:66:semi-primary}, Theorem 3).\\

  Since $R$ is semiprimary, it is perfect on both sides. By the above it also has finite global dimension. Since every Gorenstein flat has finite projective dimension, we have $Flat = \mathcal{GF}$. Similarly, $Proj = \mathcal{GP}$. Thus $Proj = Flat = \mathcal{GP} = \mathcal{GF}$. So $R$ is left GF-closed, and $\mathcal{GP} = \mathcal{GF}$ \\

The right ideal $I$ of $R$\\

$\displaystyle I = \left[
                     \begin{array}{ccc}
                       0 & \mathbb{Q} & \mathbb{R} \\
                       0 & 0 & 0 \\
                       0 & 0 & 0 \\
                     \end{array}
                   \right]/
                   \left[
                     \begin{array}{ccc}
                       0 & 0 & \mathbb{R} \\
                       0 & 0 & 0 \\
                       0 & 0 & 0 \\
                     \end{array}
                   \right]$

 is finitely generated (by the equivalence class of the matrix with $1$ in position $12$ and zeros everywhere else). \\
 Since the equivalence class of\\
 $ \left[ \begin{array}{ccc}
                       0 & 0 & 0 \\
                       0 & 0 & x \\
                       0 & 0 & 0 \\
                     \end{array}
                     \right]$

                     is in the annihilator of $I$ for any real number $x$, it follows that the annihilator of $I$ is not finitely generated (otherwise we obtain a contradiction: that $\mathbb{R}$ is finitely generated over $\mathbb{Q}$).

  So $R$ is not a right coherent ring (by \cite{lam:00:lectures}).\\

2)  Another example can be obtained by considering a two by two matrix triangular ring $S$ with the diagonal entries from a field $K$, and with the other non zero entries from a field $Q$ with $K \subseteq Q$ of left dimension $m < \infty$ but such that $Q$ has infinite right $K$-dimension. This triangular matrix ring is left hereditary and left perfect of finite global dimension. Using a modified version of Lam's argument (\cite{lam:00:lectures}, page 139) one can show that the annihilator of

$$\displaystyle \left[
                     \begin{array}{cc}
                       0 & 1 \\
                       0 & 0 \\
                     \end{array}
                   \right]$$

 is not a finitely generated right ideal. So the ring $S$ is not right coherent.


3) The previous examples can be used to construct more examples of left GF-closed rings that are not right coherent, and such that over these rings every Gorenstein projective is Gorenstein flat and every Gorenstein flat has finite Gorenstein projective dimension.\\

Let $R_1 $ be right coherent and left n-perfect and let $R$ be as in the example (1) above. Then let $\Gamma = R_1 \times R$. Since both $R_1$ and $R$ are left GF-closed, so is $\Gamma$ (by \cite{bennis:07:stronglygorenstein}).\\
- We show that $\mathcal{GP}(\Gamma) = \mathcal{GP}(R_1) \times \mathcal{GP}(R)$ \\
By \cite{bennis:07:stronglygorenstein}, any $\Gamma$-module $M$ is of the form $M_1 \oplus M_2$ with $M_1$ and $R_1$ module and $M_2$ an $R$-module. \\
Let $M_1$ be a Gorenstein projective $R_1$-module and let $M_2$ be an $R$- Gorenstein projective module. Then there are exact complexes of projective $R_1$, respectively, $R$-modules, $\textbf{P}_i = \ldots \rightarrow P_{1,i} \rightarrow P_{0,i} \rightarrow P_{-1,i} \rightarrow \ldots$ with $M_i = Ker (P_{1,i} \rightarrow P_{0,i})$. Then $\textbf{P} = \ldots \rightarrow P_1 \rightarrow P_0 \rightarrow P_{-1} \rightarrow \ldots$ is an exact complex with each $P_i = P_{1,i} \oplus P_{2,i}$ a projective $\Gamma$-module. \\
Let $Q = Q_1 \oplus Q_2$ be a projective $\Gamma$-module; then $Q_1$ is a projective $R_1$-module and $Q_2$ is a projective $R$-module. By \cite{bennis:07:stronglygorenstein} we have that 
$Hom(\textbf{P}, Q) \simeq Hom(\textbf{P}_1, Q_1) \oplus Hom(\textbf{P}_2, Q_2)$. So $Hom(\textbf{P}, Q) $ is an exact complex.\\

Also, let $M = M_1 \oplus M_2$ be a Gorenstein projective $\Gamma$-module. Then there exists an exact and $Hom(-, Proj)$ exact complex of $\Gamma$ projective modules $\textbf{P} = \ldots \rightarrow P_1 \rightarrow P_0 \rightarrow P_{-1} \rightarrow \ldots$ with $M = Ker(P_1 \rightarrow P_0)$.\\

Then, by \cite{bennis:07:stronglygorenstein}, $P_i = P_{1,i} \oplus P_{2,i}$ and $\textbf{P} = \textbf{P}_1 \oplus \textbf{P}_2$ with $\textbf{P}_i = \ldots \rightarrow P_{i,1} \rightarrow P_{i,0} \rightarrow \ldots$ an exact complex of projective modules with $M_i = ker (P_{i,1} \rightarrow P_{i,0})$.\\ Let $Q = Q_1 \oplus Q_2$ be a projective $\Gamma$ module. Then $Hom(\textbf{P}, Q) \simeq Hom(\textbf{P}_1, Q_1) \oplus Hom(\textbf{P}_2, Q_2)$. It follows that $M_1$ is a Gorenstein projective $R_1$-module and $M_2$ is an $R$-Gorenstein projective module.

By \cite{bennis:07:stronglygorenstein}, we have that $\mathcal{GF}(\Gamma) = \mathcal{GF}(R_1) \times \mathcal{GF}(R)$. \\
For any $M \in \mathcal{GP}(\Gamma)$ we have $M = M_1 \oplus M_2$ with $M_1 \in \mathcal{GP}(R_1)$ and $M_2 \in \mathcal{GP}(R)$. Then $M_1 \in \mathcal{GF}(R_1)$ and $M_2 \in \mathcal{GF}(R)$, so $M \in \mathcal{GF}(\Gamma)$.\\
Let $N$ be a Gorenstein flat $\Gamma$-module. Then $N = N_1 \oplus N_2$ with $N_1 \in \mathcal{GF}(R_1)$, and $N_2 \in \mathcal{GF}(R)$. Since $Gpd{_\Gamma} N \le sup \{Gpd_{R_1}(N_1), Gpd_R(N_2) \}$ and $Gpd_{R_1}(N_1) < \infty$ and  $Gpd_R(N_2) < \infty$, it follows that $Gpd {_\Gamma} N < \infty$

\textbf{Acknowledgement} We thank Dag Oskar Madsen and Pedro Guil Asensio for their help on some of the examples in section 4.

\vspace{7mm}

\textbf{Authors}:\\

Sergio Estrada, University of Murcia, Spain, sestrada@um.es\\
Alina Iacob, Georgia Southern University, USA, aiacob@georgiasouthern.edu\\
Katelyn Yeomans, Georgia Southern University, USA, ky00362@georgiasouthern.edu


\begin{thebibliography}{1}

\bibitem{bennis:07:stronglygorenstein}
D.~Bennis and N.~Mahdou.
\newblock {Strongly Gorenstein projective, injective, and flat modules}.
\newblock {\em J. Pure Apl. Algebra}, 210:437--445, 2007

\bibitem{christensen:06:gorenstein}
L.~Christensen and A.~Frankild and H.~Holm.
\newblock {On Gorenstein projective, injective, and flat modules. A functorial description with applications}.
\newblock {\em J.  Algebra}, 302(1):231--279, 2006.


\bibitem{enochs:00:relative}
E.E. Enochs and O.M.G. Jenda.
\newblock {\em Relative Homological Algebra}.
\newblock Walter de Gruyter, 2000.
\newblock De Gruyter Exposition in Math.

\bibitem{iacob:14:gor.proj.precovers}
S. Estrada and A. Iacob and S. Odaba\c{s}\i.
\newblock Gorenstein projective and flat (pre)covers.
\newblock Submitted.


\bibitem{harada:66:semi-primary}
M. Harada.
\newblock Hereditary semi-primary rings and triangular matrix rings.
\newblock {\em Nagoya J. Math.},27(2) :463-484, 1966.


\bibitem{holm04}
H.\ Holm.
\newblock {Gorenstein homological dimensions}.
\newblock {\em J. Pure and Appl. Alg.}, 189:167--193, 2004.

\bibitem{jorgensen:05:finite flat}
 P. J{\o}rgensen.
\newblock Finite flat and projective dimension.
\newblock {\em Comm. Algebra}, 33(7):2275-2279, 2005.

\bibitem{jorgensen:07:gorproj}
 P. J{\o}rgensen.
\newblock Existence of Gorenstein projective resolutions and Tate cohomology.
\newblock {\em J. Eur. Math. Soc}, 9, 59–-76, 2007.




\bibitem{lam:00:lectures}
T. Y. Lam.
\newblock {\em Lectures on Modules and Rings}.
\newblock Springer Verlag, 1999.


\bibitem{MS}
{\sc D.\ Murfet and S.\ Salarian}, {\sl Totally acyclic complexes over noetherian schemes.} Adv.\ Math. {\bf 226} (2011), 1096--1133.

\bibitem{yangliu}
G. Yang and K. Z. Liu.
\newblock Gorenstein flat covers over GF-closed rings.
\newblock {\em Comm. Algebra}, 40:1632--1640, 2012.


\end{thebibliography}
\end{document}